\documentclass[11pt]{amsart}

\usepackage{amsfonts,amssymb,amsthm,eucal,color,bbm,verbatim,graphicx}
\usepackage[french,english]{babel}

\usepackage[margin=1.25in]{geometry}
\calclayout

\newtheorem{thm}{Theorem}
\newtheorem{cor}[thm]{Corollary}
\newtheorem{lemma}[thm]{Lemma}
\newtheorem{prop}[thm]{Proposition}

\newtheorem{conj}[thm]{Conjecture}

\newcommand{\R}{\mathbb{R}}
\newcommand{\E}{\mathbb{E}}

\newcommand{\Prob}{\mathbb{P}}
\newcommand{\N}{\mathbb{N}}

\newcommand{\Z}{\mathbb{Z}}
\newcommand{\C}{\mathbb{C}}

\newcommand{\inprod}[2]{\left\langle #1, #2 \right\rangle}

\newcommand{\abs}[1]{\left\vert #1 \right\vert}
\newcommand{\norm}[1]{\left\Vert #1 \right\Vert}

\DeclareMathOperator{\vol}{vol}

\newcommand{\Set}[2]{\left\{#1 \mathrel{} \middle| \mathrel{} #2
  \right\}}

\newcommand{\Mag}[1]{\operatorname{Mag}\left( #1 \right)}

\DeclareMathOperator{\inrad}{inrad}

\newcommand{\ind}[1]{\mathbbm{1}_{#1}}

\allowdisplaybreaks[3]

\author{Mark W.\ Meckes}

\email{mark.meckes@case.edu}

\address{Department of Mathematics, Applied Mathematics, and
  Statistics, Case Western Reserve University, 10900 Euclid Ave.,
  Cleveland, Ohio 44106, U.S.A.}

\title[Magnitude and intrinsic volumes]{On the magnitude
  and intrinsic volumes of a convex body in Euclidean space}

\date{12 October 2019}

\begin{document}

\begin{abstract}
  Magnitude is an isometric invariant of metric spaces inspired by
  category theory. Recent work has shown that the asymptotic behavior
  under rescaling of the magnitude of subsets of Euclidean space is
  closely related to intrinsic volumes.  Here we prove an upper bound
  for the magnitude of a convex body in Euclidean space in
  terms of its intrinsic volumes.  The result is deduced from an
  analogous known result for magnitude in $\ell_1^N$, via approximate
  embeddings of Euclidean space into high-dimensional $\ell_1^N$
  spaces. As a consequence, we deduce a sufficient condition for
  infinite-dimensional subsets of a Hilbert space to have finite
  magnitude.  The upper bound is also shown to be sharp to first order
  for an odd-dimensional Euclidean ball shrinking to a point; this
  complements recent work investigating the asymptotics of magnitude
  for large dilatations of sets in Euclidean space.
\end{abstract}

\maketitle

\section{Introduction and main results}

Magnitude is an isometric invariant of metric spaces defined by
Leinster \cite{Leinster} based on category-theoretic considerations.
It is an abstract notion of the size of a metric space, which in some
ways serves as an ``effective number of points'' in the space.
Magnitude turns out to encode many classical invariants from integral
geometry and geometric measure theory, including volume, capacity,
dimension, and surface area.  See \cite{LeMe-survey} for a survey of
connections between magnitude and geometry. In other directions,
magnitude has connections to graph invariants \cite{Leinster-graph},
theoretical ecology \cite{SoPo,LeMe-div}, and homology theory
\cite{HeWi,LeSh,Otter,Hepworth,Cho}.

The purpose of this note is to show that the magnitude of a convex
body (i.e., a nonempty compact convex set) $K$ in the $d$-dimensional Euclidean
space $\ell_2^d$ is bounded above by a particular linear combination
of the intrinsic volumes of $K$ (Theorem \ref{Thm:l2-magnitude}).  The
only such sets whose magnitudes are known explicitly are Euclidean
balls for odd $d$, and even in those cases the statement for arbitrary
odd $d$ is quite complicated \cite{BaCa,Willerton-ball1} (see Theorem
\ref{Thm:Willerton-ball} below).

The upper bound in Theorem \ref{Thm:l2-magnitude} can be used to show
that certain infinite-dimensional compact sets in a Hilbert space have
finite magnitude, specifically, so-called Gaussian bounded sets
(Corollary \ref{Cor:Hilbert}). The bound is also sharp to first order
for odd-dimensional Euclidean balls with small radius, as shown in
Theorem \ref{Thm:ball-derivative}.  These results can be used to
clarify the asymptotic behavior of the magnitude of a convex body in
$\ell_2^d$ as it shrinks to a point (Corollaries \ref{Cor:point} and
\ref{Cor:mean-width}).

Magnitude can be defined in several equivalent ways (see
\cite{LeMe-survey}). For the purposes of this paper the following will
suffice.  A metric space $(X,d)$ is called \textbf{positive definite}
if, for each $n \in \N$ and each collection of distinct
$x_1, \dots, x_n \in X$, the matrix
$\bigl(e^{-d(x_i,x_j)}\bigr)_{1 \le i,j \le n}$ is positive definite.
Every subset of $L_p$ for $1\le p \le 2$ is positive definite; this of
course includes subsets of $\ell_p^d$, the space $\R^d$ equipped with
the $\ell_p$ metric for $1 \le p \le 2$.  (See
\cite[Theorem 3.6]{MM-pdms} for a broad list of positive definite
metric spaces.)  If $(X,d)$ is a compact positive definite metric
space, then the \textbf{magnitude} of $X$ is
\begin{equation}
  \label{Eq:mag-def}
  \Mag{X} = \sup\Set{\frac{\left(\sum_{i=1}^n
        w_i\right)^2}{\sum_{i,j=1}^n e^{-d(x_i,x_j)}w_i w_j}}{n \in
    \N, \ x_1, \dots, x_n \in X, \ 0 \neq w \in \R^n}.
\end{equation}
It is an open question whether this supremum is finite for every
compact positive definite metric space.

For $0 \le k \le d$, the intrinsic volumes of a convex body $K
\subseteq \ell_2^d$ can be defined by the Kubota formula
\begin{equation}
  \label{Eq:Kubota}
  V_k(K) = \binom{d}{k} \frac{\omega_d}{\omega_k \omega_{d-k}}
  \int_{\mathrm{Gr}_{d,k}} \vol_k\bigl(\pi_P(K)\bigr) \ d\mu_{d,k}(P),
\end{equation}
where $\mathrm{Gr}_{d,k}$ is the Grassmann manifold of $k$-dimensional
subspaces of $\R^d$, $\mu_{d,k}$ denotes the rotation-invariant probability
measure on $\mathrm{Gr}_{d,k}$, $\pi_P$ denotes the orthogonal
projection onto $P$, and
\[
  \omega_n = \frac{\pi^{n/2}}{\Gamma \left(1 + \frac{n}{2}\right)}
\]
is the volume of the unit ball in $\ell_2^n$; see e.g.\ \cite[p.\
222]{SchWe}. 

The normalization of the intrinsic volumes is chosen such that if
$T:\ell_2^d \to \ell_2^N$ is an isometric embedding and
$K \subseteq \ell_2^d$ is a convex body, then $V_k(T(K)) = V_k(K)$ for
all $0 \le k \le d$. It follows that $V_k(K)$ is well-defined for any
finite-dimensional convex body $K$ in a Hilbert space $\mathcal{H}$.
For a general convex body $K \subseteq \mathcal{H}$, we define
\[
V_k(K) = \sup \Set{V_k(L)}{L \subseteq K \text{ is a
    finite-dimensional convex body}}.
\]

The first main result of this paper is the following.

\begin{thm}
  \label{Thm:l2-magnitude}
  If $K \subseteq \ell_2^d$ is a convex body, then
  \begin{equation}
    \label{Eq:l2-magnitude}
      \Mag{K} \le \sum_{k=0}^d \frac{\omega_k}{4^k} V_k(K),
  \end{equation}
  with equality if $d = 1$.
\end{thm}

Theorem \ref{Thm:l2-magnitude} can be compared to the erstwhile
conjecture (see \cite{LeWi}, \cite[Conjecture 3.5.10]{Leinster}) that
if $K \subseteq \ell_2^d$ is a convex body, then
\begin{equation}
  \label{Eq:convex-conjecture}
  \Mag{K} = \sum_{k=0}^d \frac{1}{k! \omega_k} V_k(K).
\end{equation}
The explicit computation of magnitude for odd-dimensional Euclidean
balls in \cite{BaCa} showed that \eqref{Eq:convex-conjecture} is false
for $d \ge 5$ (although it does hold if $K$ is a three-dimensional
Euclidean ball).  Since that work, attention has turned to weaker
versions of this conjecture, in particular the question of whether
intrinsic volumes can be recovered from the magnitude function,
defined below. We note that the first two terms of the right hand
sides of both \eqref{Eq:l2-magnitude} and \eqref{Eq:convex-conjecture}
are $1 + \frac{1}{2} V_1(K)$; after that the coefficients in the upper
bound in \eqref{Eq:l2-magnitude} are larger.

A metric space $(X,d)$ is said to be of \textbf{negative type} if
$tX : = (X,td)$ is positive definite for every $t > 0$; examples
include every subset of $L_p$ for $1 \le p \le 2$.  The
\textbf{magnitude function} of a compact metric space of negative type
$X$ is the function $t \mapsto \Mag{tX}$ for $t > 0$.  Since $V_k$ is
homogeneous of degree $k$, \eqref{Eq:l2-magnitude} is equivalent to
the following polynomial upper bound on the magnitude function of a
convex body $K \subseteq \ell_2^d$:
\begin{equation}
  \label{Eq:magnitude-poly}
    \Mag{tK} \le \sum_{k=0}^d \frac{\omega_k}{4^k} V_k(K) t^k
\end{equation}
for $t \ge 0$.

As a consequence of Theorem \ref{Thm:l2-magnitude}, we are able to
show for the first time that some infinite-dimensional subsets of a
Hilbert space have finite magnitude.

\begin{cor}
  \label{Cor:Hilbert}
  Let $X$ be a compact subset of a Hilbert space
  $\mathcal{H}$, and let $K$ be the closed convex hull of $X$.  If
  $V_1(K) < \infty$, then $X$ has finite magnitude.
\end{cor}

Convex bodies $K \subseteq \mathcal{H}$ with $V_1(K) < \infty$ are
referred to as GB (\textbf{Gaussian bounded}) convex bodies due to
their connection with the theory of Gaussian random processes
\cite{Dudley,Chevet} (see also \cite{Vitale1} for discussion,
examples, and further references).  The only previously known examples
of infinite-dimensional metric spaces with finite magnitude were
subsets of infinite-dimensional boxes
$\prod_{i=1}^\infty [0,a_i] \subseteq \ell_1$ for
$\sum_{i=1}^\infty a_i < \infty$; see the first open problem in
\cite[Section 5]{LeMe-survey}.

Another consequence of Theorem \ref{Thm:l2-magnitude} is a new proof,
and partial extension to infinite dimensions, of a surprisingly
nontrivial fact about the behavior of the magnitude when a set in
Euclidean space shrinks to a point.

\begin{cor}
  \label{Cor:point}
  Let $X$ be a nonempty compact subset of a Hilbert space $\mathcal{H}$, and
  let $K$ be the closed convex hull of $X$.  If $V_1(K) < \infty$,
  then
  \[
    \lim_{t \to 0^+} \Mag{tX} = 1.
  \]
  In particular, this holds for any nonempty compact set
  $X \subseteq \ell_2^d$.
\end{cor}

The finite-dimensional case of Corollary \ref{Cor:point} was first
proved in \cite[Theorem 1]{BaCa} using Fourier-analytic techniques and
a potential-theoretic characterization of magnitude in $\ell_2^d$ from
\cite{MM-mag-etc}.  It was reproved in \cite[Corollary
1]{Willerton-ball1} using an exact expression for the magnitude of
odd-dimensional Euclidean balls (stated as Theorem
\ref{Thm:Willerton-ball} below).  The corresponding result for subsets
of $\ell_1^d$ is much simpler (see \cite[Proposition
4.4]{LeMe-survey}).  On the other hand, there exists a six-point
metric space of negative type $(X,d)$ for which
$\lim_{t \to 0^+} \Mag{tX} = 6/5$ \cite[Example 2.2.8]{Leinster}.

Theorem \ref{Thm:l2-magnitude} and Corollaries \ref{Cor:Hilbert} and
\ref{Cor:point} will be proved in section \ref{S:l2-magnitude}.

For odd-dimensional Euclidean balls, the upper bound in Theorem
\ref{Thm:l2-magnitude} --- and therefore the previously conjectured
formula \eqref{Eq:convex-conjecture} --- also captures the correct
first-order behavior of the magnitude function as $t \to 0$, as the
following theorem shows.

\begin{thm}
  \label{Thm:ball-derivative}
  Suppose that $d$ is odd, and let $B_2^d$ denote the Euclidean unit
  ball in $\ell_2^d$. Then
  \[
    \lim_{t \to 0^+} \frac{\Mag{tB_2^d} - 1}{t} = \frac{1}{2} V_1(B_2^d).
  \]
\end{thm}

Theorem \ref{Thm:ball-derivative} was conjectured by Simon Willerton
in response to a question by the author, on the basis of computer
calculations using the results of \cite{Willerton-ball1}. The result 
suggests the following conjecture (which would have followed from
\eqref{Eq:convex-conjecture} if that conjecture had been true).

\begin{conj}
  \label{Con:V1}
  If $K \subseteq \ell_2^d$ is a convex body, then
  \begin{equation}
    \label{Eq:derivative}
    \lim_{t \to 0^+} \frac{\Mag{tK} - 1}{t} = \frac{1}{2} V_1(K).
  \end{equation}
\end{conj}

If $d$ is odd and $X \subseteq \ell_2^d$ is the closure of a bounded
open set with smooth boundary, then \cite[Theorem 2]{GiGo} shows that
the magnitude function of $X$ has a meromorphic continuation to $\C$.
Corollary \ref{Cor:point} implies that this continuation does not have
a pole at $0$, and is thus analytic in a neighborhood of $0$.  In
particular, if $d$ is odd and $K$ is a smooth convex body with
nonempty interior, then the limit in \eqref{Eq:derivative} does exist.

Theorems \ref{Thm:l2-magnitude} and \ref{Thm:ball-derivative} can be
combined to prove a partial result in the direction of Conjecture
\ref{Con:V1}. The following result extends to the infinite-dimensional
setting if $K$ is a GB body, but for simplicity we state it here in
finite dimensions only.  We denote by $A_{d,k}$ the set of
$k$-dimensional affine subspaces of $\R^d$, and for $E \in A_{d,k}$ we
let $\inrad(K \cap E)$ be the largest radius of a $k$-dimensional
Euclidean ball contained in $K \cap E$.

\begin{cor}
  \label{Cor:mean-width}
  There is an absolute constant $c > 0$ such that if $K \subseteq \ell_2^d$ is
  a convex body, then
  \begin{equation*}
    \begin{split}
     c \max_{\substack{1 \le k \le d,\\k \text{ odd}}} 
      \sup_{E \in A_{d,k}} \sqrt{k} \inrad(K \cap E)
      & \le \liminf_{t \to 0^+}
      \frac{\Mag{tK}-1}{t} \\
      & \le \limsup_{t \to 0^+} \frac{\Mag{tK}-1}{t} \le \frac{V_1(K)}{2}.
    \end{split}
  \end{equation*}
\end{cor}

The limits inferior and superior in Corollary \ref{Cor:mean-width} are
necessarily both homogeneous of degree $1$ as functions of $K$, as are
the stated upper and lower bounds. It is not a priori obvious, however,
that the limits inferior and superior are finite and nonzero.  We
remark that \cite[Theorem 1.1]{HeHe} proves a lower bound on intrinsic
volumes of a convex body of similar nature to the lower bound in
Corollary \ref{Cor:mean-width}.

Theorem \ref{Thm:ball-derivative} and Corollary \ref{Cor:mean-width}
will be proved in section \ref{S:ball-derivative}.

\bigskip

On the other side, for any compact $X \subseteq \ell_2^d$, $\Mag{X}
\ge \frac{\vol_d(X)}{d! \omega_d}$ \cite[Theorem 3.5.6]{Leinster} and
\begin{equation}
  \label{Eq:BaCa}
  \lim_{t \to \infty} \frac{\Mag{tX} }{t^d} = \frac{\vol_d(X)}{d!
    \omega_d}
\end{equation}
\cite[Theorem 1]{BaCa} (which was consistent with the formerly
conjectured formula \eqref{Eq:convex-conjecture}). Thus our polynomial
upper bound \eqref{Eq:magnitude-poly} captures the correct order of
growth of $\Mag{tK}$ as $t \to \infty$ when $K$ has nonempty interior,
but with the wrong constant if $K$ is greater than one-dimensional.

When $X \subseteq \ell_2^d$ is the closure of a bounded, open set with
smooth boundary and $d \ge 3$ is odd, there is the finer asymptotic
expansion
\begin{equation}
  \label{Eq:GiGo}
  \begin{split}
  \Mag{tX} = & \frac{1}{d! \omega_d} \left(\vol_d(X) t^d + \frac{d+1}{2}
    \vol_{d-1}(\partial X) t^{d-1} + \frac{(d-1)(d+1)^2}{8}
    \left(\int_{\partial X} H \ dS \right) t^{d-2} \right) \\
  & + O\bigl(t^{d-3}\bigr)
  \end{split}
\end{equation}
as $t \to \infty$ \cite{GiGo}.  Here $H$ is the mean curvature on
$\partial X$ and $S$ is the surface area measure.  When $K
\subseteq \ell_2^d$ is a convex body with nonempty interior
and smooth boundary, \eqref{Eq:GiGo} becomes
\begin{equation*}
%  \label{Eq:GiGo-convex}
  \begin{split}
  \Mag{tK} = & \frac{1}{d! \omega_d} \left(V_d(K) t^d + (d+1)
    V_{d-1}(K) t^{d-1} + \frac{\pi}{4}(d+1)^2
    V_{d-2}(K) t^{d-2} \right) \\
  & + O\bigl(t^{d-3}\bigr).
  \end{split}
\end{equation*}
This implies that $V_{d-1}(K)$ and $V_{d-2}(K)$ can also be recovered
from the magnitude function of $K$. It also shows that, although the
upper bound in \eqref{Eq:magnitude-poly} only matches the
$t \to \infty$ asymptotics of the magnitude function of $K$ in a rough
sense, the dependence of the three top-order terms on $K$ is,
intriguingly, correct up to scalar multiples.  However, the next term
in the asymptotic expansion \eqref{Eq:GiGo} turns out not to be a
multiple of an intrinsic volume \cite{Goffeng-email}.

\section{Proofs of Theorem \ref{Thm:l2-magnitude} and its corollaries,
  and some related questions}
\label{S:l2-magnitude}

Theorem \ref{Thm:l2-magnitude} follows from a similar result for
magnitude of convex bodies in $\ell_1^N$. For $0 \le k \le N$, the
$\ell_1$ intrinsic volumes of a convex body
$K \subseteq \ell_1^N$ are defined by
\[
V_k'(K) = \sum_{P \in \mathrm{Gr}'_{N,k}} \vol_k\bigl(\pi_P(K)\bigr),
\]
where $\mathrm{Gr}'_{N,k}$ denotes the set of $k$-dimensional
coordinate subspaces of $\R^N$ and $\pi_P$ denotes the coordinate
projection onto $P$ \cite{Leinster-int}.  (In fact, the natural class
of sets to consider is somewhat larger than convex bodies, but this
point will not be used here.) Note that if $K$ lies in a
$d$-dimensional subspace of $\ell_1^N$, then $V_k'(K) = 0$ for
$k > d$.

\begin{thm}[{\cite[Theorem 4.6]{LeMe-survey}}]
  \label{Thm:l1-magnitude}
  If $K \subseteq \ell_1^N$ is a convex body, then
  \begin{equation}
    \label{Eq:l1-magnitude}
      \Mag{K} \le \sum_{k=0}^N \frac{1}{2^k} V_k'(K),
  \end{equation}
  with equality if $K$ has nonempty interior, or if $N = 2$.
\end{thm}

We note that, by the $\ell_1$ analogue of Steiner's formula
\cite[Theorem 6.2]{Leinster-int}, the right hand side of
\eqref{Eq:l1-magnitude} is equal to
$\vol_N \bigl(\frac{1}{2}K + [0,1]^N\bigr)$.  There does not appear to
be such a simple interpretation of the upper bound in
\eqref{Eq:l2-magnitude}.

The idea of the proof of Theorem \ref{Thm:l2-magnitude} is to
approximate the Euclidean space $\ell_2^d$ by subspaces of $\ell_1^N$
for large $N$, and show that the $\ell_1$ intrinsic volumes
approximate scalar multiples of the classical intrinsic volumes in
those subspaces.

Let $\Omega_{d,n} = (\{-1,1\}^n)^d$, equipped with the uniform
probability measure $\Prob_{d,n}$.  We will consider
$L_1(\Omega_{d,n}) = L_1(\Omega, \Prob_{d,n})$ and
$\ell_1(\Omega_{d,n}) \cong \ell_1^{2^{nd}}$, which are both the space
of functions $f:\Omega_{d,n} \to \R$ but with different norms:
\[
  \norm{f}_{L_1} =
  \E_{d,n} \abs{f} =
  \frac{1}{2^{nd}} \sum_{x \in \Omega_{d,n}} \abs{f(x)}
   = \frac{1}{2^{nd}} \norm{f}_{\ell_1}.
\]

For $1 \le i \le d$ and $1 \le j \le n$, define
$X_{i,j} = X^{(d,n)}_{i,j} : \Omega_{d,n} \to \R$ by
$X_{i,j}(x) = x_{i,j}$. Then, with respect $\Prob_{d,n}$,
$\Set{X_{i,j}}{1 \le i \le d,\ 1 \le j \le n}$ are independent,
identically distributed random variables with
$\Prob_{d,n}[X_{i,j} = 1] = \Prob_{d,n}[X_{i,j} = -1] = 1/2$.

We next define
\[
S_i^n = \frac{1}{\sqrt{n}} \sum_{j=1}^n X_{i,j}
\]
for $1 \le i \le d$, and define a linear map
$T_d^n:\ell_2^d \to L_1(\Omega_{d,n})$ by $T_d^n(e_i) = S_i^n$. We
also write $\widetilde{T}_d^n = \sqrt{\frac{\pi}{2}}2^{-nd} T_d^n$, so
that
\[
\norm{\widetilde{T}_d^n(y)}_{\ell_1} = \sqrt{\frac{\pi}{2}} \norm{T_d^n(y)}_{L_1}.
\]

To deduce Theorem \ref{Thm:l2-magnitude} from Theorem
\ref{Thm:l1-magnitude}, we will use two technical results, both of
which are applications of the central limit theorem.

\begin{lemma}
  \label{Lem:clt}
  For every $d$, $n$, and nonzero $y \in \R^d$,
  \[
    1 - \frac{4}{\sqrt{n}} \le 
    \frac{\norm{\widetilde{T}_d^n(y)}_{\ell_1}}{\norm{y}_2}
        \le 1 + \frac{4}{\sqrt{n}}.
  \]
\end{lemma}

\begin{proof}
  Without loss of generality we may assume that $\norm{y}_2 = 1$.  We have
  \[
    T_d^n(y) = \sum_{i=1}^d \sum_{j=1}^n \frac{y_i}{\sqrt{n}} X_{i,j}.
  \]
  By a version of the Berry--Esseen theorem for Lipschitz test
  functions, 
  \[
    \abs{\E_{d,n} f\bigl(T_d^n(y)\bigr) - \frac{1}{\sqrt{2\pi}}
      \int_{-\infty}^\infty f(t) e^{-t^2/2} \ dt} \le 3\sum_{i=1}^d
    \sum_{j=1}^n \abs{\frac{y_i}{\sqrt{n}}}^3 \le \frac{3}{\sqrt{n}}
  \]
  for any $1$-Lipschitz function $f:\R \to \R$. (This is essentially
  contained in the work of Esseen \cite{Esseen}; see \cite[Proposition
  2.2]{Goldstein} for an explicit statement which includes the precise
  constant used here.) In particular, letting $f(t) = \abs{t}$, this
  implies that
  \[
    \abs{\norm{T_d^n(y)}_{L_1} - \sqrt{\frac{2}{\pi}}} \le
    \frac{3}{\sqrt{n}},
  \]
  from which the lemma follows.  (The stated constant $4$ is not
  sharp.)
\end{proof}

\begin{prop}
  \label{Pro:intrinsic-volume-limit}
  If $K \subseteq \ell_2^d$ is a convex body, then for
  each $0\le k \le d$,
  \[
    \lim_{n\to \infty} V_k'\bigl(\widetilde{T}_d^n(K)\bigr) = \frac{\omega_k}{2^k} V_k(K).
  \]
\end{prop}

\begin{proof}
  The case $k=0$ is trivial, since $V_0' = V_0 = 1$ always.  Given
  $x_1, \dots, x_k \in \Omega_{d,n}$ and $f \in \ell_1(\Omega_{d,n})$,
  we denote $\pi_{x_1, \dots, x_k}(f) = (f(x_1), \dots, f(x_k))$.
  Then the $\ell_1$ intrinsic volumes of
  $X \subseteq \ell_1(\Omega_{d,n})$ can be equivalently expressed as
  \begin{equation}
    \label{Eq:V_k'}
    V_k'(X) = \frac{1}{k!} \sum_{\substack{x_1, \dots, x_k\\ \text{distinct}}}
    \vol_k\bigl(\pi_{x_1, \dots, x_k}(X)\bigr)
    = \frac{1}{k!} \sum_{x_1, \dots, x_k}
    \vol_k\bigl(\pi_{x_1, \dots, x_k}(X)\bigr).
  \end{equation}
  The restriction to distinct summands can be dropped in
  \eqref{Eq:V_k'} since if $x_1, \dots, x_k$ are not distinct, then
  the dimension of the range of $\pi_{x_1, \dots, x_k}$ is smaller
  than $k$ and $\vol_k\bigl(\pi_{x_1, \dots, x_k}(X)\bigr) = 0$.
  
  Now
  \[
    \pi_{x_1, \dots, x_k}\bigl(\widetilde{T}_d^n(y)\bigr) =
    \sqrt{\frac{\pi}{2}} 2^{-nd} \bigl(\inprod{y}{S^n(x_1)}, \dots, \inprod{y}{S^n(x_k)}\bigr),
  \]
  where $S^n(x) = (S_1^n(x), \dots, S_d^n(x)) \in
  \R^d$. Equivalently,
  \[
    \pi_{x_1, \dots, x_k}\bigl( \widetilde{T}_d^n(y)\bigr) =
    \sqrt{\frac{\pi}{2}} 2^{-nd} M_n(x_1, \dots, x_k)^t y,
  \]
  where $M_n(x_1, \dots, x_k)$ is the $d \times k$ matrix with
  entries
  $\bigl(S_i^n(x_j)\bigr)_{\substack{1\le i \le d \\ 1 \le j \le
      k}}$ and $M_n^ty$ is given by matrix multiplication. It follows
  that
  \begin{equation}
    \label{Eq:vol-pi-T}
    \vol_k\bigl( \pi_{x_1, \dots, x_k} \bigl(
    \widetilde{T}_d^n(K) \bigr)\bigr) = \left(\frac{\pi}{2} \right)^{k/2}
    2^{-ndk} \bigl(M_n^t (K)\bigr),
  \end{equation}
  Combining \eqref{Eq:V_k'} and \eqref{Eq:vol-pi-T}, we obtain
  \begin{equation}
    \label{Eq:V_k'-M}
  V_k'\bigl(\widetilde{T}_d^n(K)\bigr) = \frac{1}{k!}
  \left(\frac{\pi}{2} \right)^{k/2} \E \vol_k
  \bigl(M_n^t(K)\bigr),
  \end{equation}
  where $M_n$ is a $d\times k$ random matrix with independent entries
  each distributed as $\frac{1}{\sqrt{n}} \sum_{j=1}^n X_{1,j}$.

  The idea now is that by the central limit theorem, $M_n$ converges in
  distribution as $n \to \infty$ to a $d\times k$ random matrix $G$
  with independent standard Gaussian entries, and by a result of
  Tsirelson \cite{Tsirelson} (see also \cite{Vitale2}),
  \begin{equation}
    \label{Eq:Tsirelson}
    \E \vol_k \bigl(G^t (K)\bigr) = \frac{\omega_k k!}{(2\pi)^{k/2}}
    V_k(K).
  \end{equation}
  The application of the central limit theorem is not quite immediate,
  however, due to the unboundedness of $\vol_k\bigl(M_n^t(K)\bigr)$ as
  a function of $M_n$.  This can be handled with a standard truncation
  argument as follows.

  For a $d \times k$ matrix $A$, we write
  \[
    F(A) = \vol_k\bigl(A^t(K)\bigr) = \sqrt{\det(A^t A)}
    \vol_k\bigl(\pi_{C(A)}(K)\bigr),
  \]
  where $C(A)$ denotes the column space of $A$.  There exists an
  $R > 0$ such that $K$ is contained in a Euclidean ball of radius
  $R$; thus $\vol_k \bigl(\pi_{C(A)}(K)\bigr) \le R^k \omega_k$ for
  every $A$.  It follows that for each $D > 0$,
  \[
    F_D(A) = \vol_k\bigl(\pi_{C(A)}(K)\bigr) \min\left\{\sqrt{\det(A^t
        A)}, D \right\}
  \]
  is a bounded, continuous function of $A$. We have
  \begin{multline}
    \label{Eq:truncation}
    \abs{\E F(M_n) - \E F(G)} \\
    \le
    \abs{\E F(M_n) - \E F_D(M_n)}
    + \abs{\E F_D(M_n) - \E F_D(G)}
    + \abs{\E F_D(G) - \E F(G)}.
  \end{multline}
  The central limit theorem implies that
  \begin{equation}
    \label{Eq:clt}
    \lim_{n\to \infty} \abs{\E F_D(M_n) - \E F_D(G)} = 0
  \end{equation}
  for each $D>0$.

  Hadamard's inequality \cite[Theorem 7.8.1]{HoJo} implies
  that $\det (A^t A) \le \prod_{j=1}^k \norm{a_j}_2^2$, where $a_j$
  denotes the $j^{\mathrm{th}}$ column of $A$. It follows that
  \[
    \E \det (M_n^t M_n) \le \E \prod_{j=1}^k \norm{m_j}_2^2 =
    \prod_{j=1}^k \E \norm{m_j}_2^2 = d^k,
  \]
  and so 
  \begin{align*}
    \abs{\E F(M_n) - \E F_D(M_n)} & \le R^k \omega_k \E \left[\sqrt{\det
      (M_n^t M_n)} \ind{\sqrt{\det (M_n^t M_n)} > D}\right] \\
    & \le R^k \omega_k \sqrt{\E \det (M_n^t M_n)} 
      \sqrt{\Prob \left[ \sqrt{\det (M_n^t M_n)} > D \right]}
    \le \frac{R^k \omega_k d^{k}}{D} 
  \end{align*}
  for each $n$ by the Cauchy--Schwarz and Markov inequalities. The
  same argument applies to the last term in \eqref{Eq:truncation}
  (which could also be more simply handled with the monotone or
  dominated convergence theorem).

  By \eqref{Eq:truncation} and \eqref{Eq:clt} we now have that
  \[
    \limsup_{n \to \infty} \abs{\E F(M_n) -\E F(G)} \le \frac{2 R^k
      \omega_k d^{k}}{D}
  \]
  for each $D > 0$. Letting $D \to \infty$ we conclude that
  $\E F(M_n) \xrightarrow{n \to \infty} \E F(G)$, which, by
  \eqref{Eq:V_k'-M} and \eqref{Eq:Tsirelson}, completes the proof.
\end{proof}

\begin{proof}[Proof of Theorem \ref{Thm:l2-magnitude}]
  Recall that the Lipschitz distance between two homeomorphic metric
  spaces $(X,d_X)$ and $(Y,d_Y)$ is defined to be
  \[
  \inf\Set{\abs{\log \operatorname{dil}(f)} + \abs{\log
      \operatorname{dil}(f^{-1})}}{f : X \to Y \text{ bi-Lipschitz}},
  \]
  where
  \[
  \operatorname{dil}(f) = \sup_{x_1 \neq x_2}
      \frac{d_Y\bigl(f(x_1),f(x_2)\bigr)}{d_X(x_1,x_2)} 
  \]
  and $\operatorname{dil}(f^{-1})$ is defined similarly.
  
  If $X \subseteq \ell_2^d$ is a fixed compact set (equipped with the
  $\ell_2^d$ metric), then Lemma \ref{Lem:clt} implies that the metric
  spaces $\widetilde{T}_n^d(X) \subseteq \ell_1(\Omega_{d,n})$
  (equipped with the $\ell_1(\Omega_{d,n})$ metric) converge to $X$ in
  the Lipschitz distance when $n \to \infty$.  This implies that
  $\widetilde{T}_n^d(X) \xrightarrow{n \to \infty} X$ also in the
  Gromov--Hausdorff distance (see \cite[Section 3.A]{Gromov}).

  Magnitude is lower semicontinuous with respect to the
  Gromov--Hausdorff topology on the collection of positive definite
  metric spaces \cite[Theorem 2.6]{MM-pdms}. It follows that
  \begin{equation*}
%    \label{Eq:liminf}
    \Mag{X} \le \liminf_{n\to \infty} \Mag{\widetilde{T}_d^n(X)}.
  \end{equation*}
  If $K \subseteq \ell_2^d$ is a convex body, Theorem
  \ref{Thm:l1-magnitude} then implies that
  \[
  \Mag{K} \le \liminf_{n \to \infty} \sum_{k=0}^{2^{nd}} \frac{1}{2^k}
  V'_k\bigl(\widetilde{T}_d^n(K)\bigr) = \liminf_{n \to \infty} \sum_{k=0}^d \frac{1}{2^k}
  V'_k\bigl(\widetilde{T}_d^n(K)\bigr).
  \]
  The upper bound in \eqref{Eq:l2-magnitude} now follows from
  Proposition \ref{Pro:intrinsic-volume-limit}.
  
  Equality for $d=1$ follows from the known formula
  $\Mag{[0,\ell]} = 1 + \frac{1}{2} \ell$ \cite[Theorem 7]{LeWi}.
\end{proof}

\bigskip

Theorem \ref{Thm:l2-magnitude} and its proof highlight some open
questions about continuity properties of magnitude.  As noted in the
statement of Theorem \ref{Thm:l1-magnitude}, the upper bound in
\eqref{Eq:l1-magnitude} is actually equal to $\Mag{K}$ if
$K\subseteq \ell_1^N$ is $N$-dimensional; the upper bound for
lower-dimensional sets in $\ell_1^N$ follows by approximation by
$N$-dimensional sets.  As we have seen, Theorem \ref{Thm:l2-magnitude}
is similarly deduced by approximating $K \subseteq \ell_2^d$ by
subsets of $\ell_1^N$ which are homeomorphic to $K$.

The $t \to \infty$ asymptotics of the magnitude function in
\eqref{Eq:BaCa} show that if $K$ is greater than one-dimensional, then
the upper bound on $\Mag{tK}$ in \eqref{Eq:magnitude-poly} must be
strict for large enough $t$.  This implies that somewhere in the
string of approximations leading from Theorem \ref{Thm:l1-magnitude}
for $N$-dimensional sets in $\ell_1^N$ to Theorem
\ref{Thm:l2-magnitude} for convex bodies in $\ell_2^d$, magnitude must
fail to be continuous.  In particular, at least one of the two
following statements must be false:
\begin{itemize}
\item For each $N$, if $K \subseteq \ell_1^N$ is a convex body
  then $\Mag{K} = \sum_{k=0}^N 2^{-k} V_k'(K)$ (\cite[Conjecture 
3.4.10]{Leinster}, \cite[Conjecture 4.5]{LeMe-survey}). Equivalently,
  magnitude is continuous with respect to the Hausdorff
  distance on the collection of convex bodies in $\ell_1^N$.

\item For each $d$, magnitude is continuous with respect to the
  Hausdorff distance on the collection of $d$-dimensional convex
  bodies in $L_1$.
\end{itemize}

Magnitude is known to be continuous on the collection of
$d$-dimensional convex bodies in any fixed $d$-dimensional subspace of
$L_1$ \cite[Theorem 4.15]{LeMe-survey}. Moreover, the known examples
of discontinuity of magnitude all involve change of topology. This
includes the six-point space from \cite[Example 2.2.8]{Leinster}
discussed above shrinking to a one-point space, as well as the
approximation of a sphere in Euclidean space by spherical shells
\cite{GiGo-cafe,Willerton-email}.  Available evidence is thus in favor
of the second statement above (although it is possible that both
statements are false).  In fact, we conjecture the following stronger
statement:

\begin{conj}
  \label{Con:continuity}
  Let $(X,d_X)$ be a compact metric space of negative type.  Then
  magnitude is continuous with respect to the Lipschitz distance on the
  family of metric spaces $(Y,d_Y)$ of negative type which are
  bi-Lipschitz equivalent to $(X,d_X)$.
\end{conj}

As noted above, Conjecture \ref{Con:continuity} and known results
would show that \cite[Conjecture 3.4.10]{Leinster} and
\cite[Conjecture 4.5]{LeMe-survey} are false for convex bodies
in $\ell_1^N$ without interior.

\begin{proof}[Proof of Corollary \ref{Cor:Hilbert}]
  If $Y$ is any compact positive definite metric space and
  $\emptyset \neq X \subseteq Y$, then
  \begin{equation}
    \label{Eq:monotone}
    1 \le \Mag{X} \le \Mag{Y};
  \end{equation}
  this follows immediately from our definition \eqref{Eq:mag-def} of
  magnitude.  It therefore suffices here to prove that $\Mag{K} <
  \infty$. 

  Let $\Set{x_n}{n \in \N}$ be a countable dense subset of $K$, and
  let $K_n$ be the intersection of $K$ with the linear span of
  $\{x_1, \dots, x_n\}$. Then $K_n \xrightarrow{n \to \infty} K$ in
  the Hausdorff distance, and \cite[Corollary 2.7]{MM-pdms} implies
  that $\Mag{K} = \lim_{n \to \infty} \Mag{K_n}$.

  As a consequence of the Alexandrov--Fenchel inequalities, $V_k(K_n)
  \le \frac{1}{k!} V_1(K_n)^{k}$ for every $k$ and $n$ (see
  \cite[Theorem 2]{McMullen}), and therefore by Theorem
  \ref{Thm:l2-magnitude},
  \[
    \Mag{K_n} \le \sum_{k=0}^n \frac{\omega_k}{4^k k!} V_1(K_n)^k
    \le \sum_{k=0}^\infty \frac{\omega_k}{4^k k!} V_1(K)^k.
  \]
  We conclude that
  \begin{equation}
    \label{Eq:V1-bound}
    \Mag{K} \le \sum_{k=0}^\infty \frac{\omega_k}{4^k k!} V_1(K)^k <
    \infty.
    \qedhere
  \end{equation}
\end{proof}

\begin{proof}[Proof of Corollary \ref{Cor:point}]
  Define the function
  \[
    f(t) = \sum_{k=0}^\infty \frac{\omega_k V_1(K)^k}{4^k k!} t^k.
  \]
  This power series converges for every $t \in \R$. From
  \eqref{Eq:monotone} and \eqref{Eq:V1-bound} it follows that
  \[
    1 \le \Mag{tX} \le \Mag{tK} \le f(t).
  \]
  Since $f(0) = 1$, this implies the corollary.
\end{proof}

\section{Proofs of Theorem \ref{Thm:ball-derivative} and Corollary
  \ref{Cor:mean-width}}
\label{S:ball-derivative}

Theorem \ref{Thm:ball-derivative} depends on an exact combinatorial
formula for the magnitude of a Euclidean ball in odd dimensions due to
Willerton \cite{Willerton-ball1}. To state it, we first need some
terminology and notation.

A \textbf{Schr\"oder path} is a finite directed path in $\Z^2$ in
which each step with starting point $(x,y) \in \Z^2$ is either an
\textbf{ascent} to $(x+1, y+1)$, a \textbf{descent} to $(x+1,y-1)$, or
a \textbf{flat step} to $(x+2,y)$.  For $k \ge 0$, a \textbf{disjoint
  $k$-collection} is a family of Schr\"oder paths from $(-i,i)$
to $(i,i)$ for each $0 \le i \le k$, such that no node in $\Z^2$ is
contained in two of the paths.  (Since all nodes of the paths have an
even sum of coordinates, it follows that the paths do not cross.)  We
denote by $X_k$ the set of all disjoint $k$-collections, and by
$X_k^j$ the set of disjoint $k$-collections with exactly $j$ flat
steps.  The set $X_k^0$ consists of a single collection, denoted
$\sigma^k_{\mathrm{roof}}$ in \cite{Willerton-ball1}, in which for
each $i$, the $i^\mathrm{th}$ path consists of $i$ ascents followed by
$i$ descents.

For a collection $\sigma \in X_k$ we write $\tau \in \sigma$ if $\tau$
is a step in one of the paths in $\sigma$.  For an indeterminate $t$ define
\[
w_j(\tau) = \begin{cases}
  1 & \text{if $\tau$ is an ascent,} \\
  t & \text{if $\tau$ is a flat step,}\\
  y+1-j & \text{if $\tau$ is a descent from height $y$ to height $y-1$.}
\end{cases}
\]

\begin{thm}[{\cite[Corollary 27]{Willerton-ball1}}]
  \label{Thm:Willerton-ball}
  Let $d = 2m+1$ be odd. Then
  \[
  \Mag{t B_2^d} = \frac{\displaystyle \sum_{\sigma \in X_{m+1}} \prod_{\tau \in
      \sigma} w_2(\tau)}{\displaystyle d ! \sum_{\sigma \in X_{m-1}} \prod_{\tau \in
        \sigma} w_0(\tau)}
  \]
  for all $t > 0$.
\end{thm}

\begin{proof}[Proof of Theorem \ref{Thm:ball-derivative}]
  First note that by the Kubota formula \eqref{Eq:Kubota},
  \begin{equation}
    \label{Eq:ball-V1}
  V_1(B_2^d) = \frac{(2m+1) \sqrt{\pi} \Gamma (m+1)}{\Gamma \left(m +
      \frac{3}{2}\right)} = 2 \binom{m-\frac{1}{2}}{m}^{-1},
  \end{equation}
  where $\binom{x}{k} = \frac{x (x-1) \cdots (x-k+1)}{k!}$ denotes the
  generalized binomial coefficient for $x \in \R$ and $k$ a
  nonnegative integer (with the convention that $\binom{x}{0} = 1$).

  Now write
  \[
  N(t) = \sum_{\sigma \in X_{m+1}} \prod_{\tau \in \sigma} w_2(\tau)
  \qquad \text{and} \qquad D(t) = \sum_{\sigma \in X_{m-1}}
  \prod_{\tau \in \sigma} w_0(\tau).
  \]
  Willerton showed in \cite[Theorem 28]{Willerton-ball1} that $N(0) =
  d! D(0)$. We wish to compute
  \begin{equation}
    \label{Eq:quotient-rule}
  \left.\frac{d}{dt} \Mag{tB_2^d}\right\vert_{t=0} = \frac{N'(0) D(0)
    - N(0)D'(0)}{d! D(0)^2} = \frac{N'(0) - d! D'(0)}{N(0)}.
  \end{equation}
  We have that
  \begin{equation}
    \label{Eq:at-0}
    \begin{split}
      N(0) & = \sum_{\sigma \in X^0_{m+1}} \prod_{\tau \in \sigma}
      w_2(\tau) = \prod_{\tau \in \sigma^{m+1}_{\mathrm{roof}}} w_2(\tau), \\
      N'(0) & = t^{-1} \sum_{\sigma \in X^1_{m+1}} \prod_{\tau \in \sigma} w_2(\tau), \\
      D'(0) & = t^{-1} \sum_{\sigma \in X^1_{m-1}} \prod_{\tau \in \sigma} w_0(\tau). \\
    \end{split}
  \end{equation}
  It is easy to give an explicit expression for $N(0)$, but it is more
  convenient here to leave it in the form above.  
  
  We instead begin by simplifying the right hand side of
  \eqref{Eq:quotient-rule} via the same trick used in
  \cite{Willerton-ball1} to show $N(0) = d! D(0)$.  Namely, each
  $\sigma \in X_{m-1}$ gives rise to a $\mu(\sigma) \in X_{m+1}$ by
  shifting all paths up two units, adding ascents from $(-i,i)$ to
  $(-i+1,i+1)$ and descents from $(i-1,i+1)$ to $(i,i)$ for
  $1 \le i \le m$, and finally adding a path from $(-(m+1),m+1)$ to
  $(m+1,m+1)$ consisting of $m+1$ ascents followed by $m+1$ descents
  (see \cite[Figure 4]{Willerton-ball1}).  Then $\mu(\sigma)$ has the
  same number of flat steps as $\sigma$, and
  \[
  \prod_{\tau \in \mu(\sigma)} w_2(\tau) = d! \prod_{\tau \in
    \sigma} w_0(\tau).
  \]
  It therefore follows from \eqref{Eq:quotient-rule} and \eqref{Eq:at-0} that
  \[
    \left.\frac{d}{dt} \Mag{tB_2^d}\right\vert_{t=0} = \sum_{\sigma
      \in X^1_{m+1} \setminus \mu(X^1_{m-1})} \frac{t^{-1} \prod_{\tau
        \in \sigma} w_2(\tau)}{\prod_{\tau \in
        \sigma^{m+1}_{\mathrm{roof}}} w_2(\tau)}.
  \]

  For $1 \le p \le k$ and $0 \le q \le k-p$, let $\sigma^k_{p,q}$
  denote the disjoint $k$-collection described as follows: the
  $p^{\mathrm{th}}$ path consists of $p-1$ ascents, one flat step, and
  $p-1$ descents.  For $p+1 \le i \le p+q$, the $i^\mathrm{th}$ path
  consists of $i-1$ ascents, one descent, one ascent, and $i-1$
  descents.  For $i<p$ and $i > p+q$, the $i^\mathrm{th}$ path
  consists of $i$ ascents followed by $i$ descents. (See Figure
  \ref{Fig:sigma-4-2-1}.)  
  \begin{figure}
    \centering
    \includegraphics[width=2in]{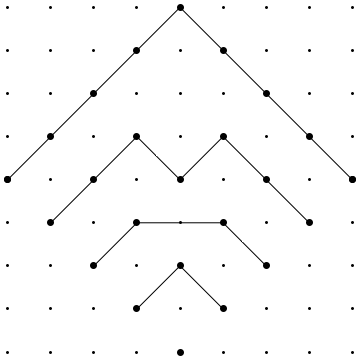}
    \caption{The disjoint $4$-collection $\sigma^4_{2,1}$.
    \label{Fig:sigma-4-2-1}}
  \end{figure}
  It is not hard to show that
  \[
    X^1_k = \Set{\sigma^k_{p,q}}{1 \le p \le k, \ 0 \le q \le k-p}.
  \]
  Moreover,
  \[
    X^1_{m+1} \setminus \mu(X^1_{m-1}) = \Set{\sigma^{m+1}_{1,q}}{0
      \le q \le m-1}
    \cup \Set{\sigma^{m+1}_{p,m+1-p}}{1 \le p \le m+1},
  \]
  where the parameter ranges are chosen so that this is a disjoint
  union.  We therefore have that
  \begin{equation}
    \label{Eq:2-sums-weights}
    \left.\frac{d}{dt} \Mag{tB_2^d}\right\vert_{t=0} = \sum_{q=0}^{m-1}
      \frac{t^{-1} \prod_{\tau \in \sigma^{m+1}_{1,q}}
      w_2(\tau)}{\prod_{\tau \in \sigma^{m+1}_{\mathrm{roof}}} w_2(\tau)}
    + \sum_{p = 1}^{m+1} \frac{t^{-1} \prod_{\tau \in \sigma^{m+1}_{p,m+1-p}}
      w_2(\tau)}{\prod_{\tau \in \sigma^{m+1}_{\mathrm{roof}}}
      w_2(\tau)}.
  \end{equation}
  By considering only which descents in $\sigma^{m+1}_{p,q}$ are not
  in $\sigma^{m+1}_{\mathrm{roof}}$, and vice versa, we find that
  \[
    \frac{t^{-1} \prod_{\tau \in \sigma^{m+1}_{p,q}}
      w_2(\tau)}{\prod_{\tau \in \sigma^{m+1}_{\mathrm{roof}}}
      w_2(\tau)} = \frac{\prod_{j=1}^q [2(p+j)-2]}{\prod_{j=0}^q [2(p+j)
      - 1]}.
  \]

  With some algebraic manipulation, the right hand side of
  \eqref{Eq:2-sums-weights} becomes
  \[
  \sum_{q = 0}^{m-1} \binom{q+\frac{1}{2}}{q}^{-1} +
  \binom{m+\frac{1}{2}}{m}^{-1} \sum_{k=0}^{m} \binom{k -
    \frac{1}{2}}{k}.
  \]
  A straightforward induction on $m$ shows that 
  \[
  \sum_{k=0}^{m} \binom{k -
    \frac{1}{2}}{k} = \binom{m+\frac{1}{2}}{m}
  \]
  for each $m \ge 0$.  By \eqref{Eq:ball-V1}, it thus remains to show
  that
  \begin{equation}
    \label{Eq:1-sum}
    \sum_{q = 0}^{m-1} \binom{q+\frac{1}{2}}{q}^{-1} =
    \binom{m-\frac{1}{2}}{m}^{-1} - 1
  \end{equation}
  for $m \ge 0$.  This follows by observing that both sides of
  \eqref{Eq:1-sum} are $0$ when $m=0$, and that
  \[
  \binom{m+\frac{1}{2}}{m+1}^{-1} - \binom{m-\frac{1}{2}}{m}^{-1} =
  \binom{m+\frac{1}{2}}{m}^{-1}.
  \qedhere
  \]
\end{proof}

\begin{proof}[Proof of Corollary \ref{Cor:mean-width}]
  The upper bound follows immediately from \eqref{Eq:magnitude-poly}.
  For the lower bound, for each odd $k$ and each $k$-dimensional
  affine subspace $E$, $K$ contains an isometric copy of
  $\inrad(K \cap E) B_2^k$, and so by Theorem
  \ref{Thm:ball-derivative} and \eqref{Eq:monotone},
  \begin{equation*}
    \begin{split}
      \Mag{tK} & \ge \Mag{t \inrad(K \cap E) B_2^k} = 1 +
      \frac{V_1(B_2^k)}{2} \inrad(K \cap E) t + o(t) \\
      & \ge 1 + c \sqrt{k} \inrad(K \cap E) t + o(t).  \qedhere
    \end{split}
  \end{equation*}
\end{proof}

\section*{Acknowledgements}
This research was partially supported by Collaboration Grant \#315593
from the Simons Foundation. The author thanks Tom Leinster and Simon
Willerton for useful conversations, and the anonymous referee for
suggestions which improved the exposition of this paper.

\bibliographystyle{plain}
\bibliography{intrinsic-volumes}

\end{document}